\documentclass[12pt]{article}
\usepackage{graphicx}
\usepackage{adjustbox}
\usepackage{fancybox}
\usepackage{amssymb}
\usepackage{amsthm}
\usepackage{amsfonts,amsmath}
\usepackage{tikz} 
\usetikzlibrary{fit,positioning}
\usepackage[section]{placeins}
\usepackage{lineno}
\usepackage{caption}
\usepackage{epstopdf}
\usepackage{subcaption}
\usepackage{authblk}
\usepackage{nameref,hyperref}

\providecommand{\keywords}[1]
{
  \small	
  \textbf{\textit{Keywords:}} #1
}

\theoremstyle{plain}
\newtheorem{theorem}{Theorem}[section]

\theoremstyle{definition}

\theoremstyle{remark}

\usepackage{nameref,hyperref}
\usepackage{xurl} 

\title{A mathematical model with nonlinear relapse: conditions for a forward-backward bifurcation}

\date{October 2022}
\author[1,*,+]{Fabio Sanchez}
\author[2,+]{Jorge Arroyo-Esquivel}
\author[1,+]{Juan G. Calvo}

\affil[1]{Universidad de Costa Rica, Centro de Investigaci\'on en Matem\'atica Pura y Aplicada - Escuela de Matem\'atica, San Jos\'e, Costa Rica}

\affil[2]{Department of Mathematics, University of California Davis, CA, USA}

\affil[*]{fabio.sanchez@ucr.ac.cr}

\affil[+]{these authors contributed equally to this work}

\date{}
\begin{document}
\maketitle

\begin{abstract}
We constructed a Susceptible-Addicted-Reformed model and explored the dynamics of nonlinear relapse in the Reformed population. The transition from susceptible considered {\it at-risk} is modeled using a strictly decreasing general function, mimicking an influential factor that reduces the flow into the addicted class. The {\it basic reproductive number} is computed. Furthermore, $R_0$ determines the local asymptotically stability of the addicted-free equilibrium. Conditions for a forward-backward bifurcation were established using $R_0$ and other threshold quantities. A stochastic version of the model is presented, and some numerical examples are shown. Results showed that the influence of the temporarily reformed individuals is highly sensitive to the initial addicted population.
\end{abstract}

\keywords{
Nonlinear relapse; backward bifurcation; epidemic models; social determinants; addiction
}

\section{Introduction} \label{sec:intro}

Infectious diseases have been a burden to public health for some time. The transmission mechanisms of pathogens are mainly close contact with an infectious host, airborne, via a vector, and in some cases via contact with an infected area~\cite{brauer2001mathematical}. However, more recently, health authorities worldwide have been vigilant with the high number of mental health incidents, highlighted primarily by the Covid-19 pandemic~\cite{who}. Social factors provide a unique challenge to construct mathematical models that include social aspects not typically included in epidemic models. However, the incorporation of social determinants in these models is inherently difficult. In previous work, social determinants were introduced as ``epidemics" where transmission happened through social interactions, similar to infectious diseases. For example, a drinking dynamics model using a nonlinear system of differential equations~\cite{sanchez2007drinking}, where the ``infectious" class was the drinking population and the interaction between nondrinkers and drinkers simulated an epidemic process. We based our theoretical framework on the latter. Other models simulating social dynamics include: a bulimia model~\cite{gonzalez2003too}, drug models~\cite{song2006raves,behrens1999dynamic}, and a sex worker industry model~\cite{davidoff2006}, among others.

Mathematical models applied to infectious diseases have become common; more recently, an insurmountable number of models arose during the Covid-19 pandemic~\cite{chowell2003,sanchez2019zika,garcia2022covid} (and references therein). In general terms, when studying infectious diseases, mathematical models help understand disease transmission dynamics. Furthermore, mathematical models, in some cases, can provide insight to health authorities to construct and develop efficient public health policies~\cite{garcia2022covid}.

Modeling social interactions as epidemic processes can provide a helpful understanding of the phenomenon studied. Here, we model addiction as an infectious disease where the interactions between the non-addicted and addicted individuals can cause an ``epidemic" process and confer an ``infection". Drug addiction has been a problem worldwide for many decades~\cite{obrien2006,wise2014,cami2003}. In particular, when the crack ``epidemic" of the 1980s was in full force, the derivative of cocaine, a more pure and more expensive narcotic, led to a faster addiction and deterioration of individuals that consumed the drug~\cite{hass2009,falck2008}.

Furthermore, relapse rates of addicted individuals, especially those that used potent narcotics such as crack cocaine, methamphetamine, fentanyl, and heroin, among others, are very high~\cite{ronaldo2017,klein2021}. In the model constructed here, we looked at nonlinear relapse rates and the influence of those who recovered and want to provide support for non-addicted presumed susceptible individuals. This is done via a general function that depends on the temporarily recovered population and other parameters. 

Epidemic models have helped describe transmission dynamics of infectious pathogens and derive strategies for their control, prevention, and reduction of incidence, among others. Here, we provide a theoretical framework to study social phenomena studied via an epidemic model and highlight the sensitivity of initial conditions.

The article is organized as follows: in Section~\ref{sec:model}, we give details of the mathematical model. In Section~\ref{sec:math}, we present the mathematical analysis. In Section~\ref{sec:stochastic}, we provide a stochastic version of the model and provide some numerical examples. Finally, in Section~\ref{sec:disc}, we provide a discussion based on our results.

\section{Mathematical Model} \label{sec:model}
The model we consider is based on~\cite{sanchez2007drinking}, where authors explored the impact of nonlinear influence on drinking behavior dynamics. In our model, we consider three compartments: susceptible individuals ($S$), addicted individuals ($A$), and temporarily reformed individuals ($\Tilde{S}$). The model transitions follow the typical SIR model~\cite{brauer2001mathematical,hethcote2000mathematics}. 

The recruitment rate, $\beta$, represents the strength of social influence on susceptible ({\it at-risk}) individuals. In this context, {\it transmission} is a collective behavior rather than an individual consequence; i.e., recruitment is not typically the work of a single individual, but instead is a result of the collective influence of a group of individuals as a whole~\cite{decker1996life}. Moreover, $\kappa \in [0,1]$ denotes the cost of addiction, and $\nu \in [0,1]$ is the willingness of reformed individuals to deter {\it at-risk} individuals from addiction. We then consider a positive, strictly decreasing smooth function $g$ 
defined by:
\begin{equation} \label{eq:reducingFactor}
g_{\kappa,\nu}(\Tilde{S})=\frac{\kappa}{1+\nu \frac{\Tilde{S}}{N}},
\end{equation}
which is a reducing factor that impacts transitions from $S$ to $A$. The function $g$ represents the impact of reformed individuals in the {\it at-risk} population. Here, high values of $\nu$ imply that a large proportion of the reformed class is helping the susceptible population, considered {\it at-risk}. 

The relapse of the reformed population is possible through interactions with individuals in the addicted class considered {\it infectious}, which refers to conditions that possibly spread through a strong collective social component. In our model, individuals can temporarily recover at rate $\gamma$ and transition into the susceptible ({\it at-risk}) class ($\Tilde{S}$). Rehabilitation programs have the potential to use the social influence of reformed individuals to deter {\it at-risk} individuals from relapse. However, reformed individuals typically encounter environmental pressures that may lead to relapse. Reformed individuals can once again become addicted via interaction with individuals in the addicted class $A$, with relapse rate $\phi$, that denotes the ``social influence" of temporarily reformed individuals. Finally, individuals leave the system at rate $\mu$, typically considered the natural exit rate. 
 
The model we just described corresponds to the system of nonlinear differential equations given by:
\begin{eqnarray} \label{eqODEs}
\frac{dS}{dt}&=&\mu N-\beta g(\Tilde{S})S \frac{A}{N}-\mu S, \nonumber \\
\frac{dA}{dt}&=&\beta g(\Tilde{S})S \frac{A}{N}+\phi \Tilde{S} \frac{A}{N}-\left(\mu+\gamma \right)A, \\
\frac{d\Tilde{S}}{dt}&=&\gamma A-\phi \Tilde{S} \frac{A}{N}- \mu \Tilde{S}, \nonumber
\end{eqnarray}
where $N=S+A+\Tilde{S}$ is the total (presumed constant) population. We re-scale system \eqref{eqODEs} by substituting $s=\frac{S}{N}$, $a=\frac{A}{N}$, $\Tilde{s}=\frac{\Tilde{S}}{N}$, obtaining the equivalent system of equations:
\begin{subequations} \label{ode}
\begin{align} 
\frac{ds}{dt}&=\mu -\beta g(\Tilde{s})s a-\mu s, \\
\frac{da}{dt}&=\beta g(\Tilde{s})s a+\phi \Tilde{s}a-(\mu+\gamma)a, \label{eqn5} \\
\frac{d\Tilde{s}}{dt}&=\gamma a-\phi \Tilde{s}a-\mu \Tilde{s}.
\end{align}
\end{subequations}
It is clear that $s+a+\Tilde{s}=1$ and the reducing factor \eqref{eq:reducingFactor} is therefore given by
\begin{eqnarray} \label{equation_f}
g(\Tilde{s})=\frac{\kappa}{1+\nu \Tilde{s}} \in [0,1].
\end{eqnarray}


\section{Mathematical analysis} \label{sec:math}
We first analyze the addiction-free equilibrium, $(s_0^*, a_0^*, \Tilde{s}_0^*)=(1,0,0)$, that can be used to determine the basic reproductive number, $R_0$. In epidemiology, the basic reproductive number represents the number of secondary infections produced by an average infected individual; when this number is less than one, the disease typically dies out, while when it is greater than one, there will be an epidemic~\cite{brauer2001mathematical}. In this context, we consider $R_0$ to be a measure of the strength of the social influence of addicted individuals to recruit individuals into a vice. As we will demonstrate, $R_0>1$ implies the establishment of an infectious agent and $R_0<1$ typically implies that the number of addicted individuals decreases and goes to zero. Albeit, our model can sustain an addiction when $R_0<1$ under particular initial conditions.

\subsection{Basic reproductive number and addiction-free equilibrium}
\noindent 
We use the next generation operator method~\cite{hethcote2000mathematics} to compute $R_0$. Let
$$\mathcal{F}= \beta g(\Tilde{s}) s a +\phi \Tilde{s} a \qquad \text{and} \qquad \mathcal{V}= (\mu +\gamma) a,$$
where $\mathcal{F}$ and $\mathcal{V}$ 
contains all terms flowing into $a$ and flowing out of $a$, respectively. It holds that
$$F=\frac{\partial \mathcal{F}}{\partial a}\Bigg|_{(s_0^*, a_0^*, \Tilde{s}_0^*)} =  \beta g(0)\,\, \text{and}\,\, V=\frac{\partial \mathcal{V}}{\partial a}\Bigg|_{(s_0^*, a_0^*, \Tilde{s}_0^*)} =\mu + \gamma.$$

The basic reproductive number is then:
$$R_0=F V^{-1}=\frac{\beta g(0)}{\mu +\gamma}=\frac{\beta \kappa}{\mu+\gamma},$$
where $\frac{1}{\mu+\gamma}$ represents the average amount of time spent in the addicted class. 
We then have the following result:

\begin{theorem}
The addiction-free equilibrium is stable if and only if $R_0<1$.
\end{theorem} 

\begin{proof}
The Jacobian of system \eqref{ode} is given by:
$$J(s,a,\Tilde{s}) =
\left[ {\begin{array}{ccc}
 \frac{-\beta \kappa}{1+\nu \Tilde{s}} a -\mu & \frac{-\beta \kappa}{1+\nu \Tilde{s}} s & 
\frac{\beta \kappa}{(1+\nu \Tilde{s})^2} s a \nu  \\
  \frac{\beta \kappa}{1+\nu \Tilde{s}} a & \frac{\beta \kappa}{1+\nu \Tilde{s}} s +\phi \Tilde{s} - 
(\mu+\gamma) & \frac{-\beta \kappa}{(1+\nu \Tilde{s})^2} s a \nu+\phi a \\
  0 & \gamma - \phi \Tilde{s}  & -\phi a-\mu
 \end{array} } \right],$$

and evaluating the addicted-free equilibrium yields to:
$$J(1,0,0) =
\left[ {\begin{array}{ccc}
 -\mu & -\beta \kappa & 0  \\
0 & \beta \kappa - (\mu+\gamma) &  0   \\
  0 & \gamma & -\mu
 \end{array}} \right].$$

The eigenvalues of this matrix are:
\begin{align*}
    \lambda_1, \: \lambda_2 &= -\mu<0,\\
    \lambda_3 &= \beta \kappa-(\mu+\gamma).
\end{align*}

Note that $\lambda_3<0  \Longleftrightarrow R_0<1$, and since the equilibrium is stable if all of the eigenvalues are negative, the result holds.
\end{proof}

\subsection{Endemic equilibria}
In order to study the prevalence of addiction and endemic equilibria in our model, we need to define an analogous basic reproductive number for the reformed class, denoted by $R_\phi$, that represents the strength of social influence of addicted individuals to recruit reformed individuals back into addiction. It is defined by
$$R_\phi=\frac{\phi}{\mu+\gamma}.$$


Since reformed individuals have already been exposed to a previous addiction, we assume that the rate at which a reformed individual relapses is higher than the rate in which an {\it at-risk} individual becomes addicted; i.e., $R_\phi>R_0$. As we will show, endemic equilibria exist whenever $R_0>1$,
and under certain initial conditions when $R_0<1$ and $R_\phi>1$. 

We study two cases: $\nu=0$ (absence of willingness factor) and $0< \nu \leq 1$. In the first case, the social influence of reformed individuals on the {\it at-risk} population is absent; this is, $g(\Tilde{s})=\kappa$. These two cases allow us to explore the impact that reformed individuals have on the population dynamics. 

\subsubsection{Case \texorpdfstring{$\nu=0$:\:}: absence of willingness factor}
Solving for the endemic equilibria of system \eqref{ode} when $\nu=0$ yields to:
$$s_1^* = \frac{\mu}{\beta \kappa a+\mu}, \qquad \Tilde{s}_1^*=\frac{\gamma a}{\phi a+\mu}.$$
Substituting these values into \eqref{eqn5} yields 
the quadratic equation $x_2 a^2+x_1 a+x_0=0$, where 
the coefficients are given by:
\begin{eqnarray*}
x_2&=&R_\phi R_0, \\
x_1&=&R_0 (1-R_\phi) + R_\mu R_\phi,\\
x_0&=&R_\mu(1-R_0),
\end{eqnarray*}
where $R_\mu = \mu/(\mu+\gamma) \in (0,1)$. In this case, our system is very similar to the drinking model studied in~\cite{sanchez2007drinking}. We can construct a bifurcation diagram to analyze the stability of the endemic equilibria as a function of $R_0$ (as $\kappa$ varies). Stability depends on both the value of $R_0$ and the initial addicted population size. In Figure~\ref{alpha0}, we show a typical backward bifurcation curve, which typically occurs in systems with nonlinear relapse rates~\cite{sanchez2007drinking,hadeler1995core,xiao2010dynamics}. Furthermore, the system exhibits hysteresis; i.e., it is highly sensitive to initial conditions~\cite{hethcote2000mathematics}. After straightforward computations, the quadratic equation has a double root when $R_0$ is equal to 
\begin{equation}\label{eq:Rc}
    R_c =  R_\mu R_\phi \frac{1+R_\phi+2\sqrt{R_\phi(1-R_\mu)}}{(R_\phi-1)^2+4R_\mu R_\phi}.
\end{equation}
It is clear that there is no positive endemic equilibria if $R_0<R_c$, and it can be shown that there are two positive endemic equilibria when $R_c<R_0<1$; see~\cite{sanchez2007drinking}.

\begin{figure}[htb!]
\begin{center}
\includegraphics[width=.6\linewidth]{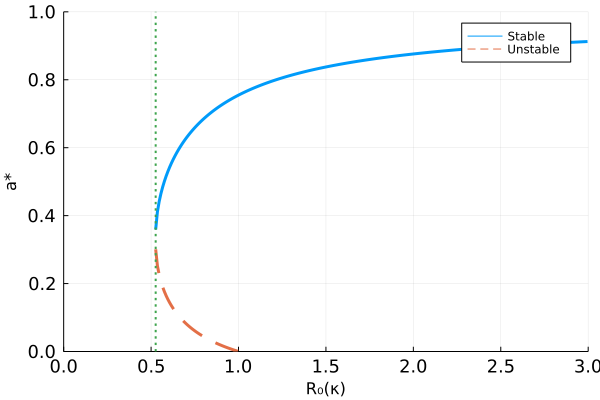}
\end{center}
\caption{Backward bifurcation with parameters $\mu=0.00015$, $\beta=0.009$, $\gamma=0.0027$, $\nu=0$, $\phi=0.005$, 
and $\kappa$ varies. The dotted vertical line represents the critical value $R_c\approx 0.52$ for which there is no positive endemic equilibria if $R_0<R_c$. There are two positive endemic equilibria when $R_c<R_0<1$.}
\label{alpha0}
\end{figure}

\subsubsection{Case \texorpdfstring{$0<\nu \leq 1$:\:} : presence of willingness factor}
Analyzing the system when $0<\nu \leq 1$ allows us to explore the impact of social influence of reformed individuals on the {\it at-risk} population. Solving for the endemic equilibria of system \eqref{ode} when $0<\nu \leq 1$ leads to 
the following:
$$s_2^* = \frac{\mu D}{\beta \kappa a+\mu D}, \qquad \Tilde{s}_2^*=\frac{\gamma a}{\phi a+\mu},$$
where $D=1 + \frac{\nu \gamma a}{\phi a+\mu}$. Substituting these values 
into Equation \ref{eqn5} yields 
the cubic equation $f(x)=0$, where 
\begin{equation} \label{eq:polynomial}
f(x) = x_3 a^3+x_2 a^2+x_1 a+x_0,    
\end{equation}
and the 
coefficients are given by:
\begin{eqnarray*}
x_3&=&R_\phi^2 R_0, \\
x_2&=&R_\phi \left[R_0 \left(1-R_\phi\right) +R_\mu( R_0+ R_\phi) +\nu R_\mu (1-R_\mu)  \right], \\
x_1&=& R_\mu \left[\nu (1-R_\mu) +R_0(1-R_\phi)+R_\phi(1- R_0) +R_\mu 
R_\phi \right], \\
x_0&=& R_\mu^2 \left(1-R_0 \right).
\end{eqnarray*}


Figure~\ref{genbif} shows a typical bifurcation exhibiting both a forward and a backward behavior for model \eqref{ode} as a function of $R_0=R_0(\kappa)$. Let $R_c$ and $R_0^*$ (with $R_c<R_0^*$) be the values of $R_0$ for which \eqref{eq:polynomial} has two double roots, similarly as \eqref{eq:Rc}. 
These two constants are thresholds that determine the number of endemic equilibria for a given value of $R_0$. We remark that: 
\begin{enumerate}
    \item If $R_0>1$, there exists at least one positive equilibrium state, since $f(0) = x_0<0$ and $f(1)= R_0 (1-R_\mu)(R_\mu+R_\phi)+R_\mu(1+R_\phi)(R_\mu+R_\phi+\nu(1-R_\mu))>0$.
    \item If $R_0<1$ and $R_\phi<1$, the coefficients of $f$ are all positive, and therefore there is no positive equilibrium state. 
\end{enumerate}

\begin{figure}[htb!]
\begin{center}
\includegraphics[width=.6\linewidth]{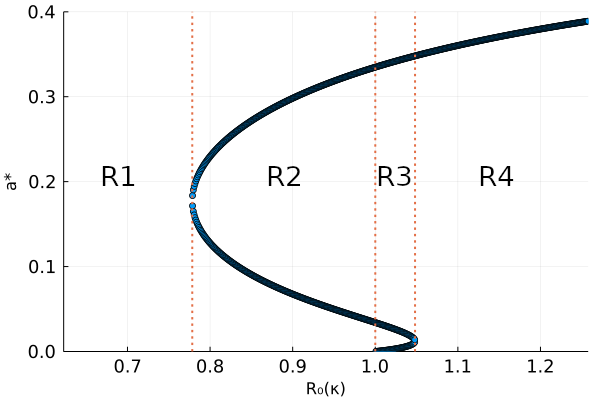}
\end{center}
\caption{Forward-backward bifurcation with parameters $\mu=0.00015$, $\gamma=0.0027$, $\beta=0.009$, $\nu=0.8$, $\phi=0.0044$ and $\kappa$ varies. The dotted lines ($R_0=R_c$, $R_0=1$ and $R_0=R_0^*$ from left to right) separate the domain into four regions: Region 1 ($R_0<R_c$) with no positive equilibria, Region 2 ($R_c<R_0<1$) with two positive equilibria, Region 3 ($1<R_0<R^*_0$) with three positive equilibria, and Region 4 ($R^*_0<R_0$) with one positive equilibria.}
\label{genbif}
\end{figure}



We now demonstrate conditions for the number of endemic equilibria as a function of $R_0$ and $R_\phi$, according to each region of interest as depicted in Figure~\ref{genbif}. In Region 2, 
two endemic equilibria exist and the system exhibits backward behavior. In this case, the initial addicted population size determines if the addicted population can establish itself or decrease to zero:

\begin{theorem} \label{prop3} If $\nu>0$, a necessary condition for two positive equilibria is $0<R_c<R_0<1$ and $R_\phi>1$.
\end{theorem}
\begin{proof}
If the system has two positive equilibria, the polynomial $f$ in Equation~\ref{eq:polynomial} has three real roots, which implies $R_0>R_c$. In addition, since $x_3>0$, only one of these equilibria has to be negative, which implies that $f(0)=\mu^2(1-R_0)>0$, which implies $R_0<1$.
\end{proof}

In Region 3, 
three endemic equilibria exist (two stable and one unstable). This implies that two end-states can occur and that the long-term population can establish itself at either a large or small size, depending on the initial addicted population size:

\begin{theorem} \label{prop4}
If $\nu>0$, a necessary condition for three positive equilibria is $1<R_0<R_0 ^*$ and $R_\phi>1$.
\end{theorem}
\begin{proof}
If the system has three positive equilibria, then the polynomial $f$ in Equation~\ref{eq:polynomial} has three real equilibria, which implies $R_c<R_0<R_0^*$. If these equilibria are all positive, in particular this implies $f(0)=\mu^2(1-R_0)<0$, which implies $R_0>1$.
\end{proof}

In Region 4, 
a unique endemic equilibrium exists and the addicted population will establish itself, regardless of the initial addicted population size:

\begin{theorem} \label{prop5}
If $\nu>0$, a sufficient condition for a unique positive equilibrium is $1<\mathcal{R}_0 ^*<\mathcal{R}_0$ and $\mathcal{R}_\phi>1$.
\end{theorem}
\begin{proof}
If $\mathcal{R}_0>\mathcal{R}_0^*$, then the polynomial $f$ in Equation~\ref{eq:polynomial} has a single real root $a_+$. Since $\mathcal{R}_0>1$, this implies that $f(0)=\mu^2(1-\mathcal{R}_0)<1$, which implies $a_+>0$.
\end{proof}

\subsection{Effect of reducing the relapse rate and the willingness factor}
We first explore the impact that relapse rate, $\phi$, has on the model. Reducing the relapse rate while still maintaining $\mathcal{R}_\phi>1$ results in a change in the behavior of the system as shown in Figure~\ref{newrho}. 
A lower relapse rate causes the bifurcation to shift to the right, where $\mathcal{R}_0<1$ guarantees a infectious-free equilibrium. This highlights the crucial role that relapse plays in addicted population dynamics. If relapse can be lowered below a critical threshold, the infectious population may be managed by just controlling $R_0$.

\begin{figure}[htb!]
\begin{center}
\includegraphics[width=.4\linewidth]{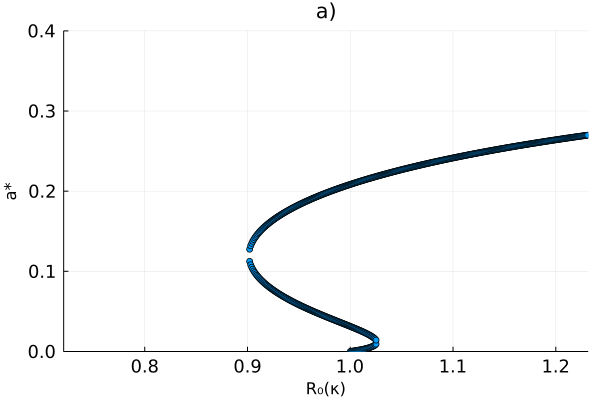}
\includegraphics[width=.4\linewidth]{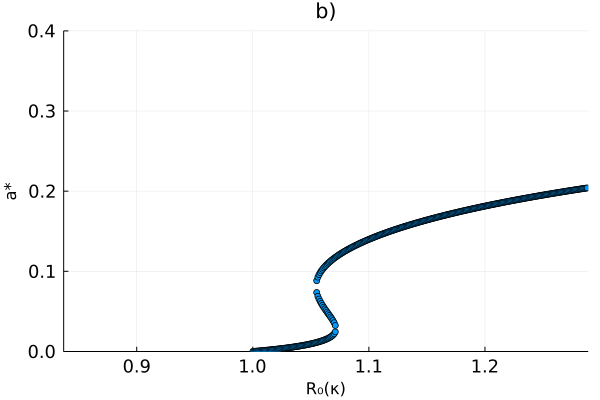}
\end{center}
\caption{Forward-backward bifurcation with parameters $\mu=0.00015$, $\gamma=0.0027$, $\beta=0.009$, $\nu=0.8$, $\kappa$ varies, and (a) $\phi=0.0044$, $\mathcal{R}_\phi=1.54$ ; (b) $\phi=0.004$, $\mathcal{R}_\phi=1.4035$.}
\label{newrho}
\end{figure}

Varying the willingness factor, $\nu$, yields significant changes in the behavior of the model; see Figure~\ref{alphamulti} where we show the bifurcation diagrams for four different values of $\nu$. Low values of $\nu$ (indicative of low interaction between the reformed and {\it at-risk} classes) yields a backward bifurcation similar to Figure~\ref{alpha0}. As $\nu$ increases, the system moves through the state shown in Figure~\ref{newrho}(b) and continues to shift to the right. When $\nu=1$, the state qualitatively resembles Figure~\ref{newrho}(b), for which $\mathcal{R}_0<1$ guarantees stability for the addiction-free equilibrium. This implies that reformed individuals helping {\it at-risk} individuals have the potential to significantly impact the long-term addicted population size, despite high relapse rates ($\mathcal{R}_\phi>1$). 

\begin{figure}[htb!]
\begin{center}
\includegraphics[width=.4\linewidth]{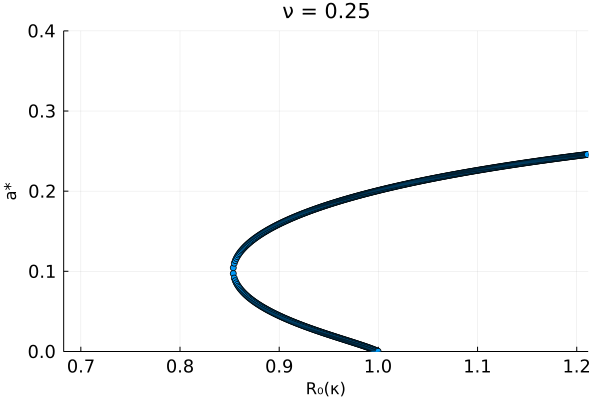}
\includegraphics[width=.4\linewidth]{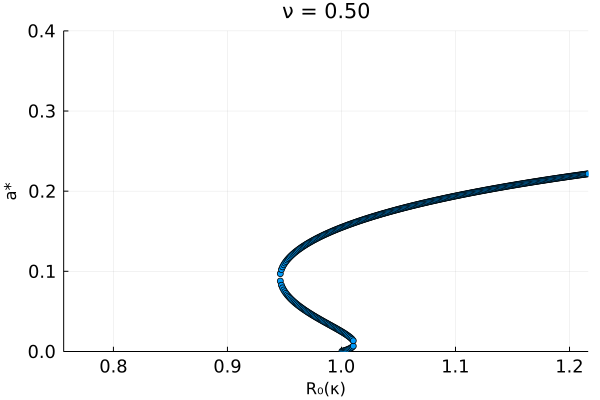}\\
\includegraphics[width=.4\linewidth]{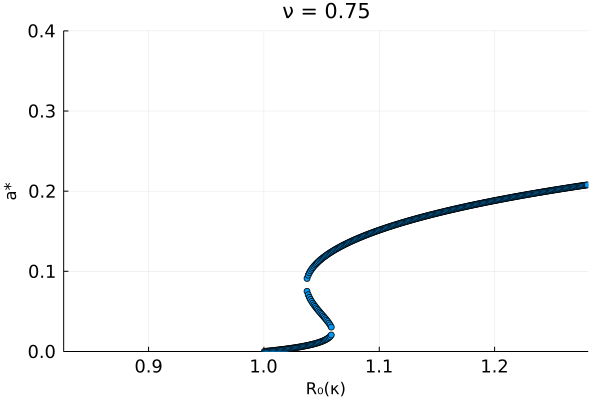}
\includegraphics[width=.4\linewidth]{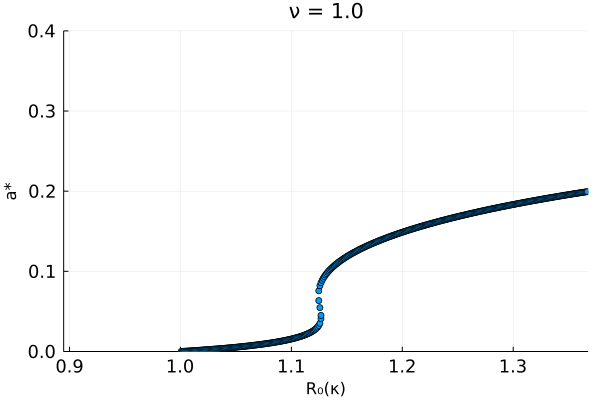}
\end{center}
\caption{Bifurcation diagrams for varying values of $\nu$ with parameters $\mu=0.00015$, $\gamma=0.0027$, $\beta=0.009$, $\phi=0.004$, $\kappa$ varies and $\nu$ as shown on each graph.}
\label{alphamulti}
\end{figure}

\section{Numerical results} \label{sec:stochastic}
Besides the deterministic model \eqref{eqODEs}, we consider also a discrete stochastic model in order to compare the behaviour of both models and the dependence on the initial conditions. For the stochastic model, probability rates between states  are given in Table \ref{tab:transitions}, which are straightforwardly obtained from \eqref{eqODEs}. We only replace the recruitment rate by a constant $\Lambda$, chosen initially as $\Lambda = N(0)\mu$, in order to essentially have a constant population. For simplicity, we take $t\in\{0,\Delta t, 2\Delta t, 3\Delta t, \ldots\}$ and the number of events on each time step $\Delta t$ is assumed to follow a Poisson distribution with mean equal to the rates shown in Table \ref{tab:transitions}.

\begin{table}[h!]
\centering 
\caption{Transition rates for the stochastic model similar to \eqref{eqODEs}\label{tab:transitions}}
\begin{tabular}{l|l|l}
Event       & Effect & Rate \\ \hline
Addiction   &  $A\leftarrow A+1$, $S\leftarrow S-1$  & $\beta g(\Tilde{S}) S A/N$ \\
Relapse &  $A\leftarrow A+1$, $\Tilde{S}\leftarrow S-1$& $\phi A \Tilde{S}/N$                             \\
Recovery    & $\Tilde{S}\leftarrow \Tilde{S}+1$, $A\leftarrow A-1$ & $\gamma A$ \\
Recruitment & $S\leftarrow S + 1$ & $\Lambda$ \\
Exit from $S$& $S\leftarrow S-1$ & $\mu S$                                          \\
Exit from $A$& $A\leftarrow A-1$ & $\mu A$                                          \\
Exit from $\Tilde{S}$& $\Tilde{S}\leftarrow \Tilde{S}-1$ & $\mu \Tilde{S}$                                 
\end{tabular}
\end{table}

Numerical simulations allow us to examine the addicted population dynamics over time for each region depicted in Figure \ref{genbif}; see Figures \ref{stochasticModel1}, \ref{stochasticModel2}, \ref{stochasticModel3}, \ref{stochasticModel4}. We consider three different populations, labeled \textit{low} ($N\approx 10,000$), \textit{medium} ($N\approx 100,000$) and \textit{high} ($N\approx 1,000,000$). Results for the stochastic model are scaled by $N(t)$. We include the deterministic solution $a(t)$ in order to compare the time series for both models.

Figure~\ref{stochasticModel1} (region 1) shows that the addicted population will, over time, decrease to the addicted-free equilibrium, for both stochastic and deterministic models, independent of the total population size. We observe more variability among simulations with smaller populations, as expected.

\begin{figure}[htb!]
\begin{center}
\includegraphics[width=.3\linewidth]{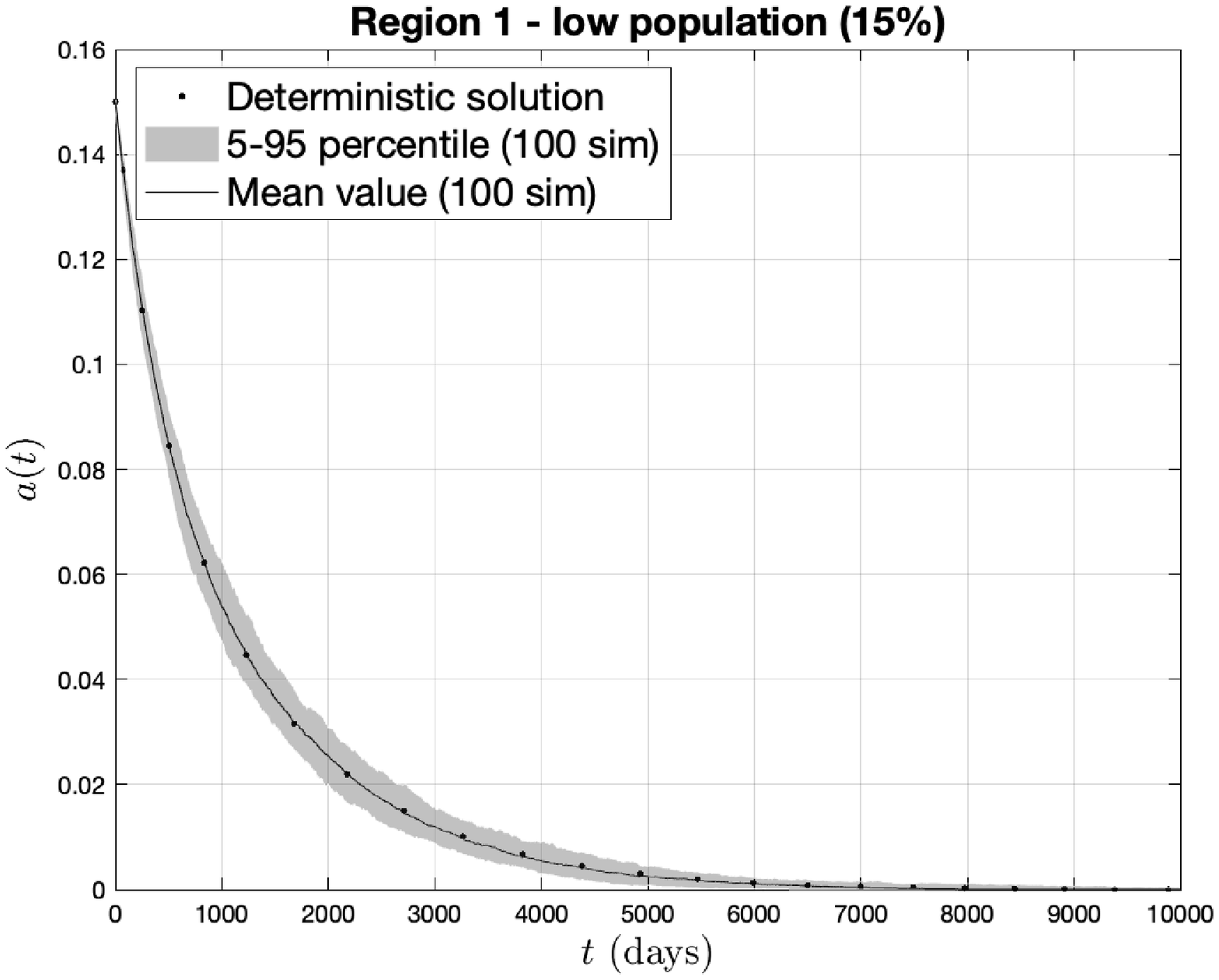} 
\includegraphics[width=.3\linewidth]{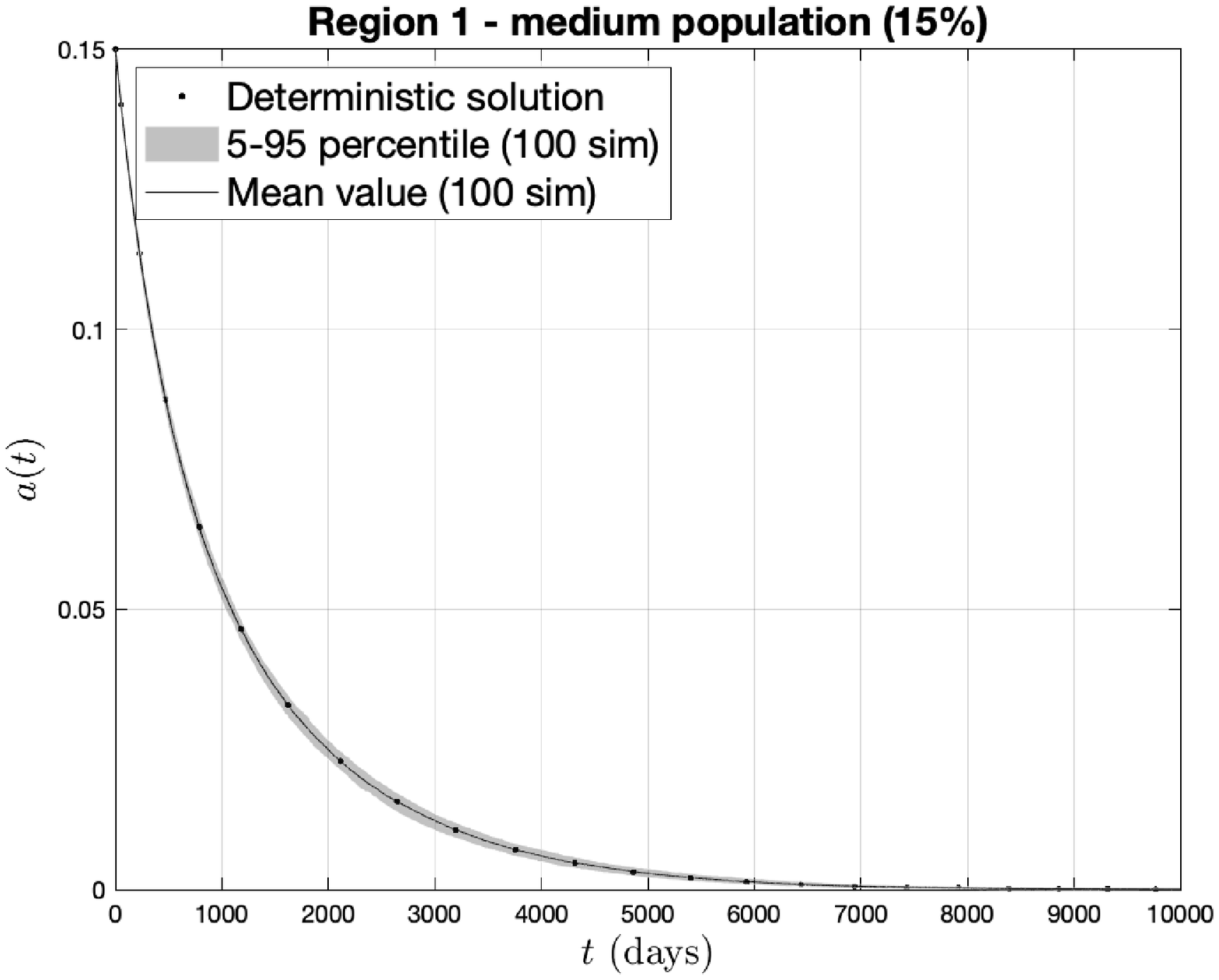} 
\includegraphics[width=.3\linewidth]{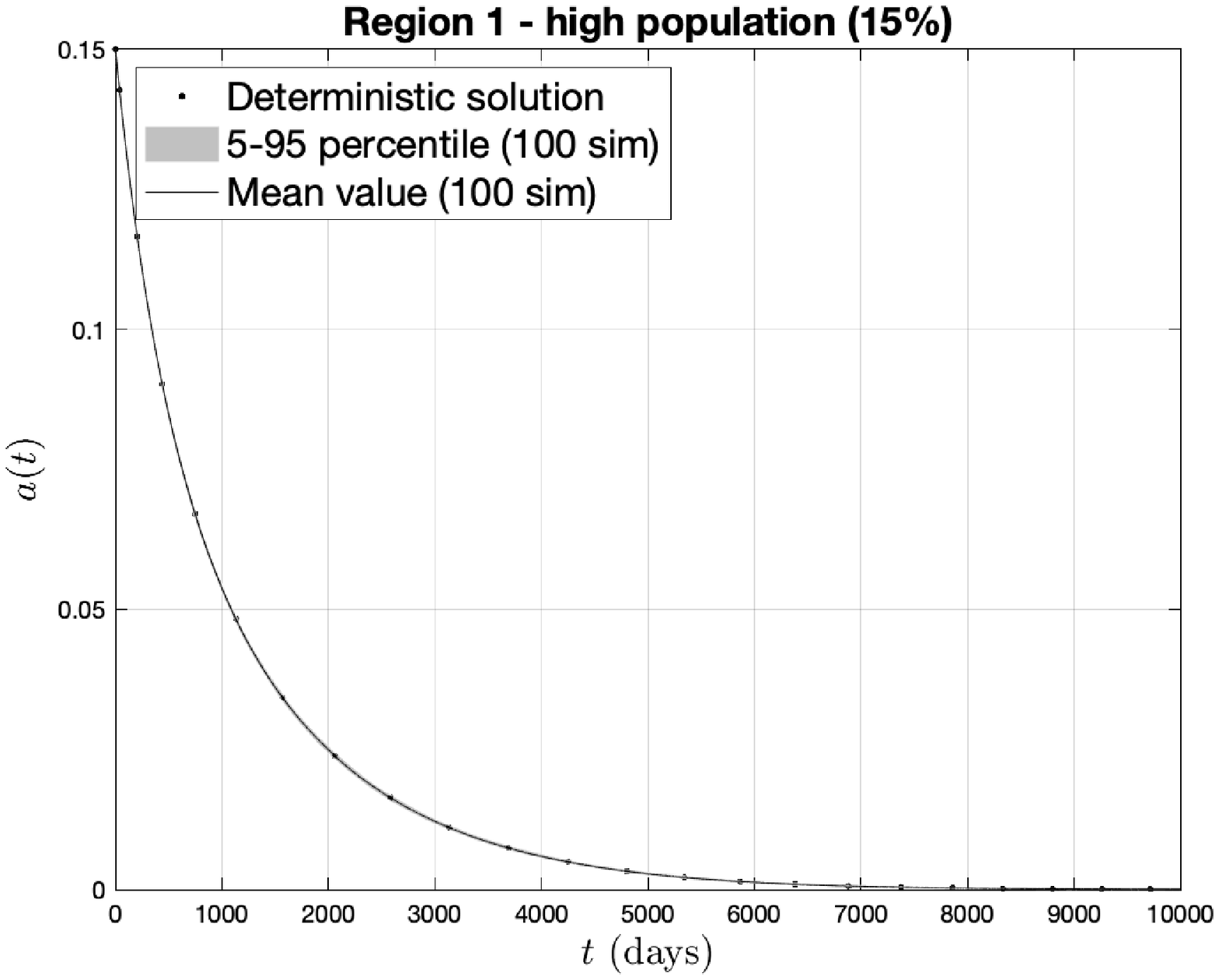} 
\end{center}
\caption{Addicted population time series for Region 1 of the forward-backward bifurcation, with parameters $\mu=0.00015$, $\beta=0.009$, $\gamma=0.0027$, $\kappa=0.2$, $\nu=0.8$, $\phi=0.0044$, $R_0=0.6315$, $R_\phi=1.5439$, $I(0)/N(0)=0.15$. We present the mean $I(t)/N(t)$ for the stochastic model (100 simulations) for (a) low, (b) medium, and (c) high populations. Gray shaded region correspond to the 5th and 95th percentiles. Blue dots correspond to the deterministic solution. \label{stochasticModel1}}
\end{figure}

Figure~\ref{stochasticModel2} (region 2) shows that the addicted population for the deterministic model will either establish itself or decrease to the addicted-free equilibrium, depending on the initial addicted population size. For the stochastic model, a similar behaviour occurs when the initial addicted population is small (1\%). Nevertheless, it the total population is small enough, it is possible to obtain an addicted-free state. We observe a large variation among simulations when there is an initial addicted compartment with 10\% of the total population in the low population case.

\begin{figure}[htb!]
\begin{center}
\includegraphics[width=.3\linewidth]{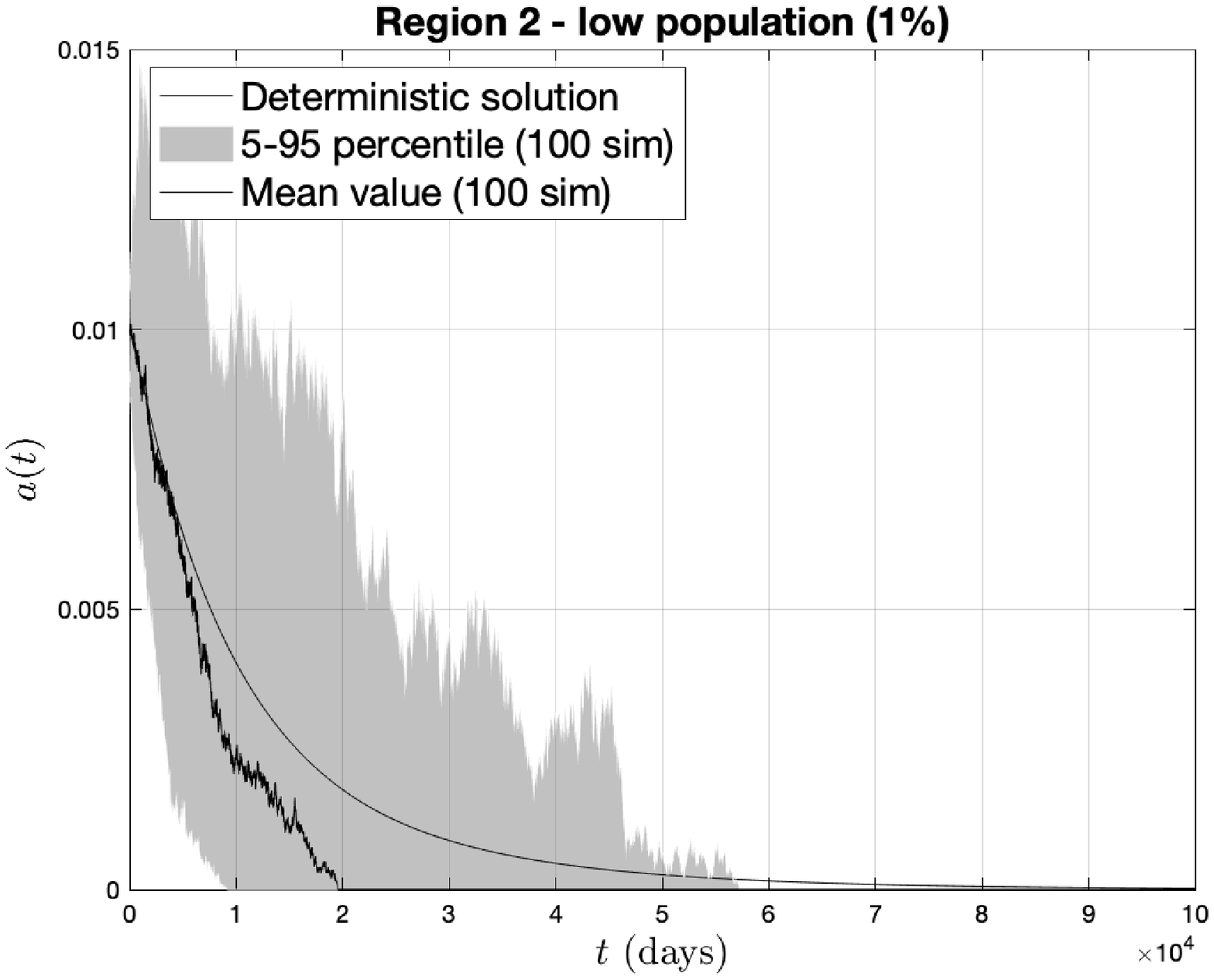}
\includegraphics[width=.3\linewidth]{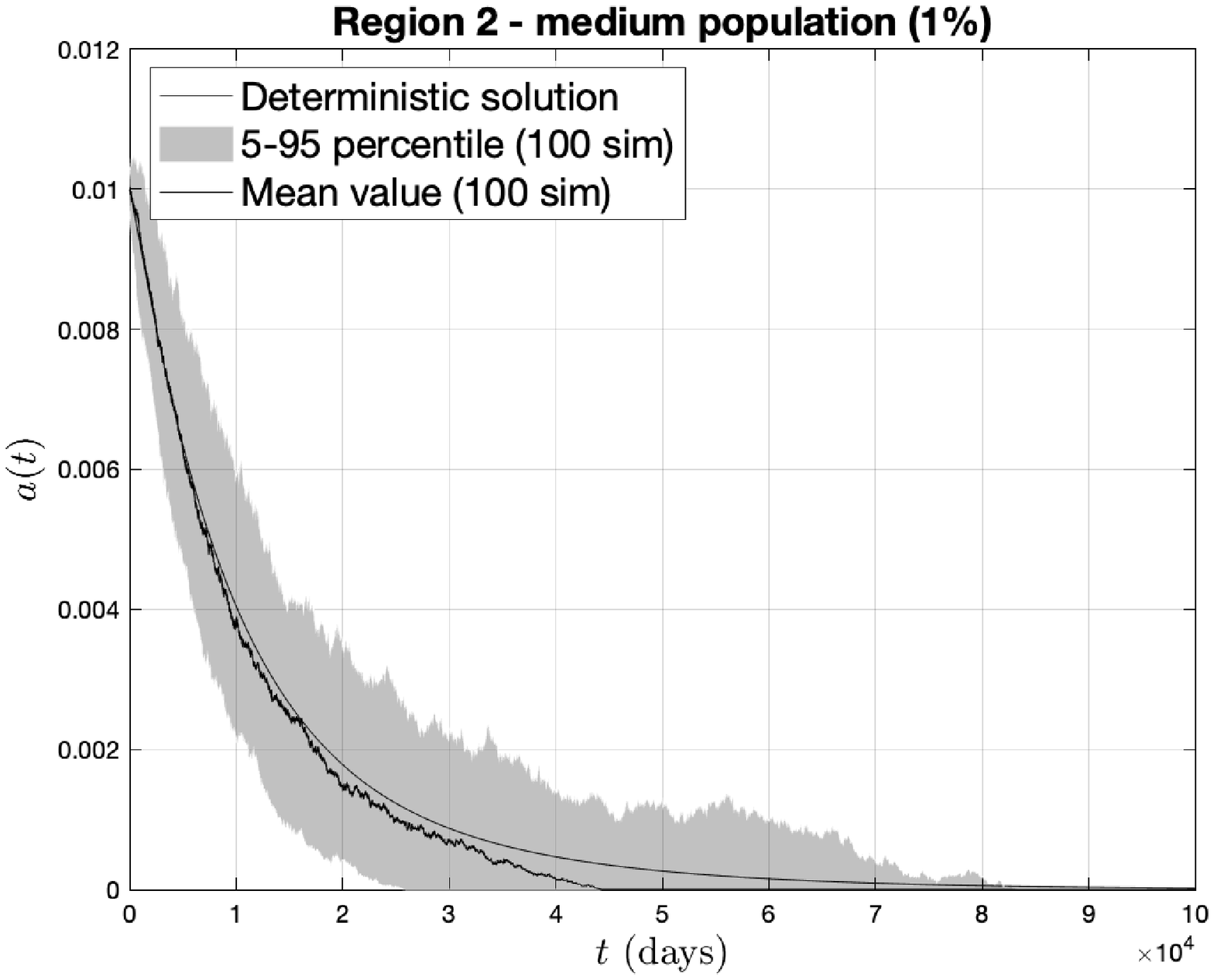}
\includegraphics[width=.3\linewidth]{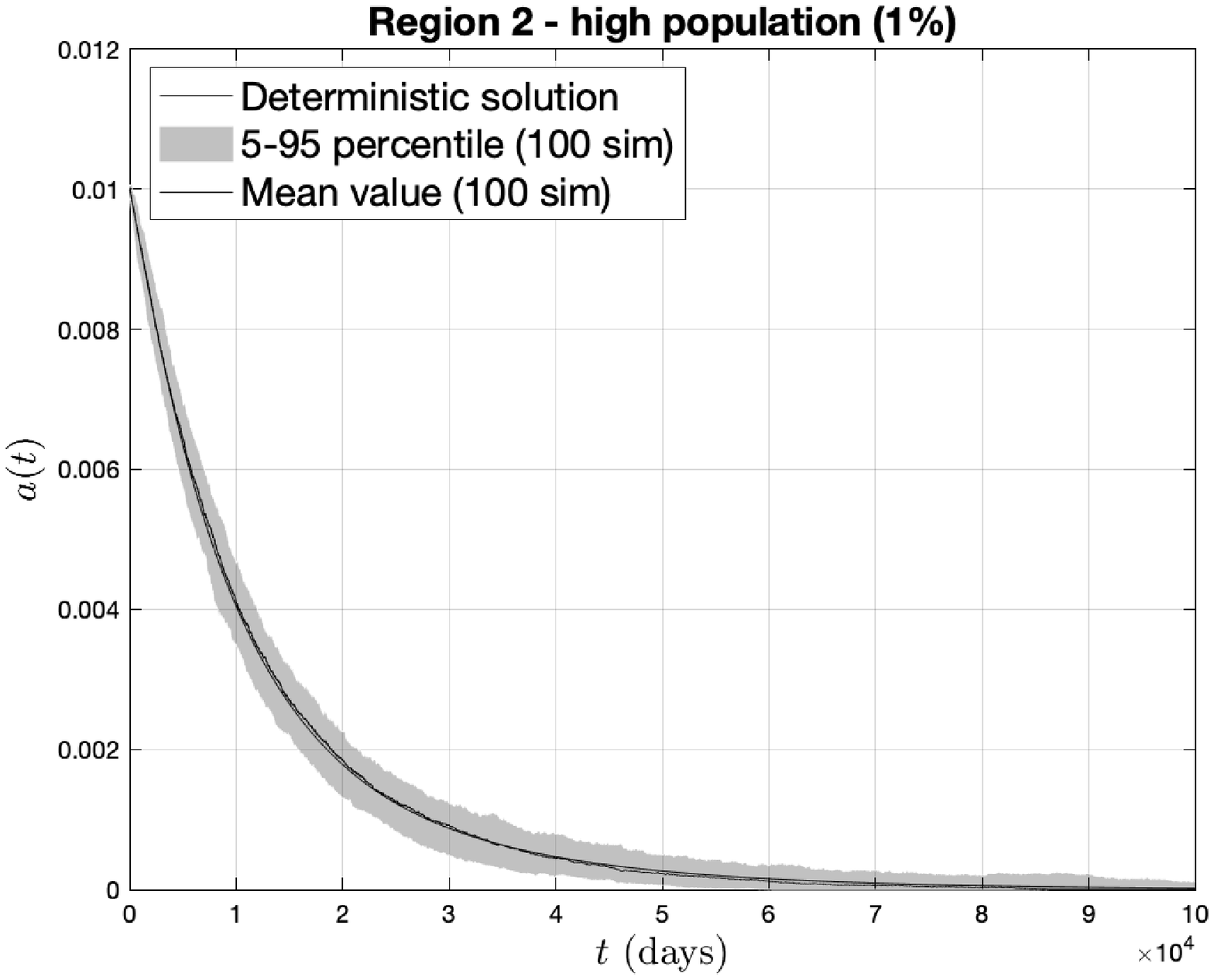}
\includegraphics[width=.3\linewidth]{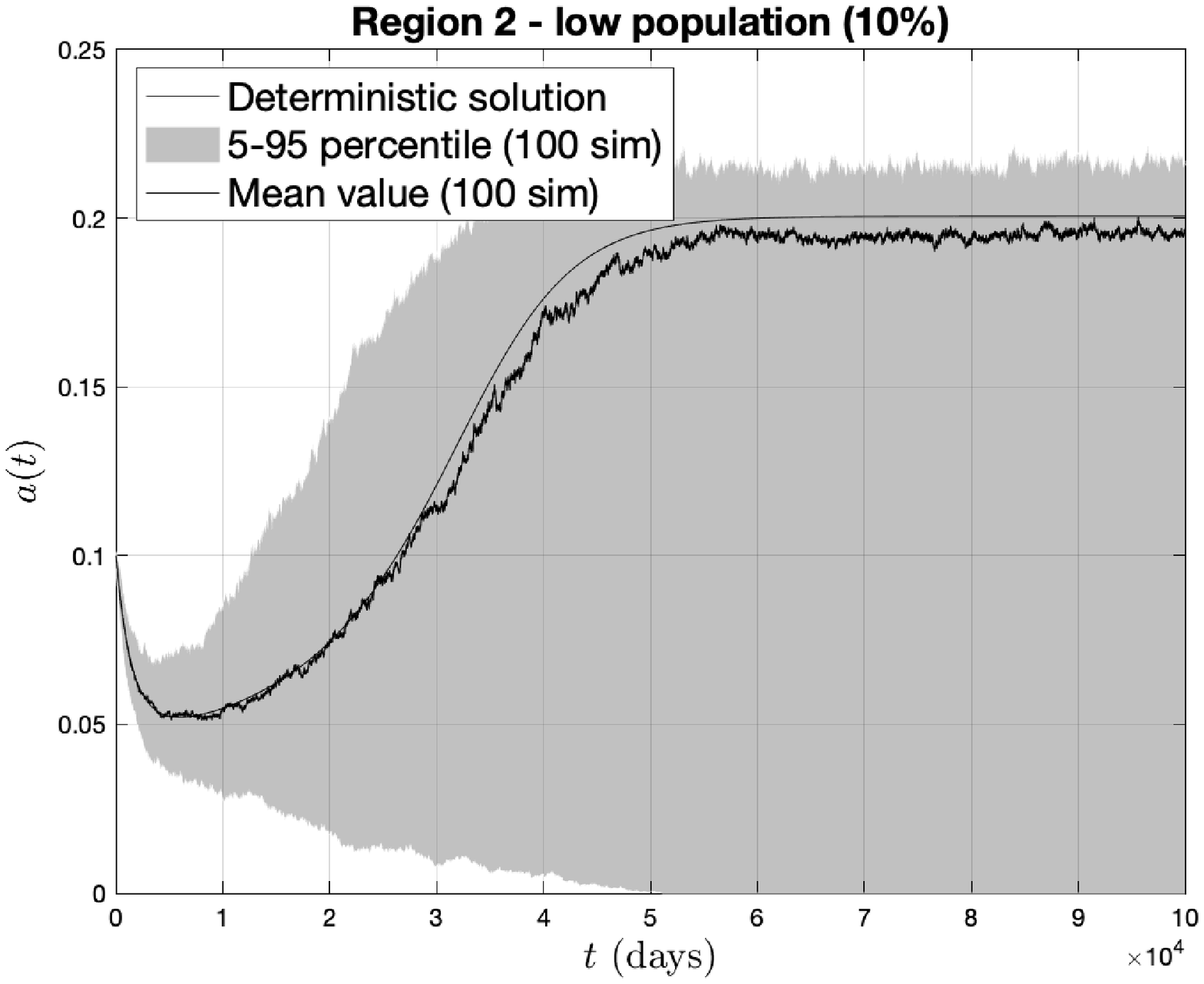}
\includegraphics[width=.3\linewidth]{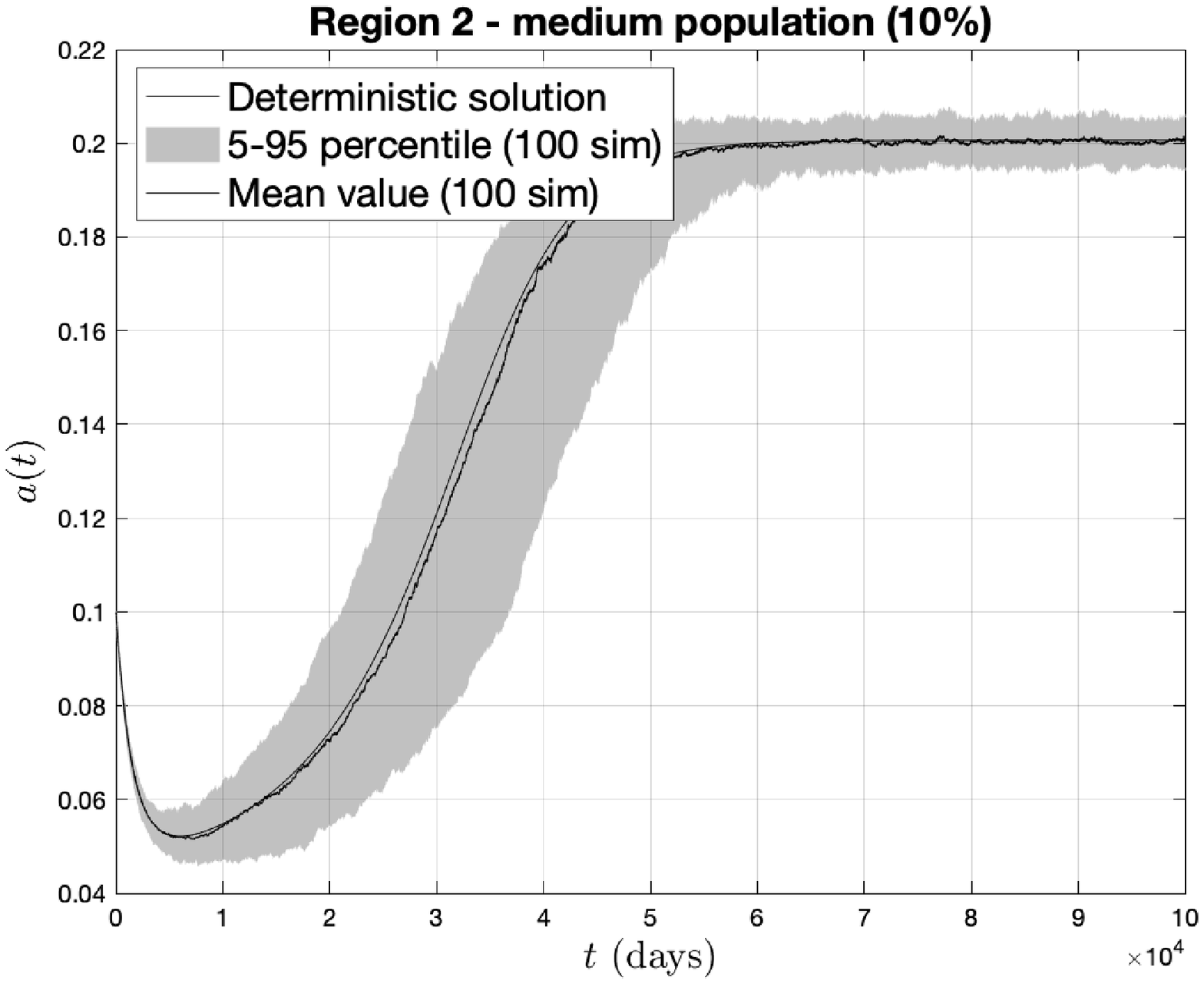}
\includegraphics[width=.3\linewidth]{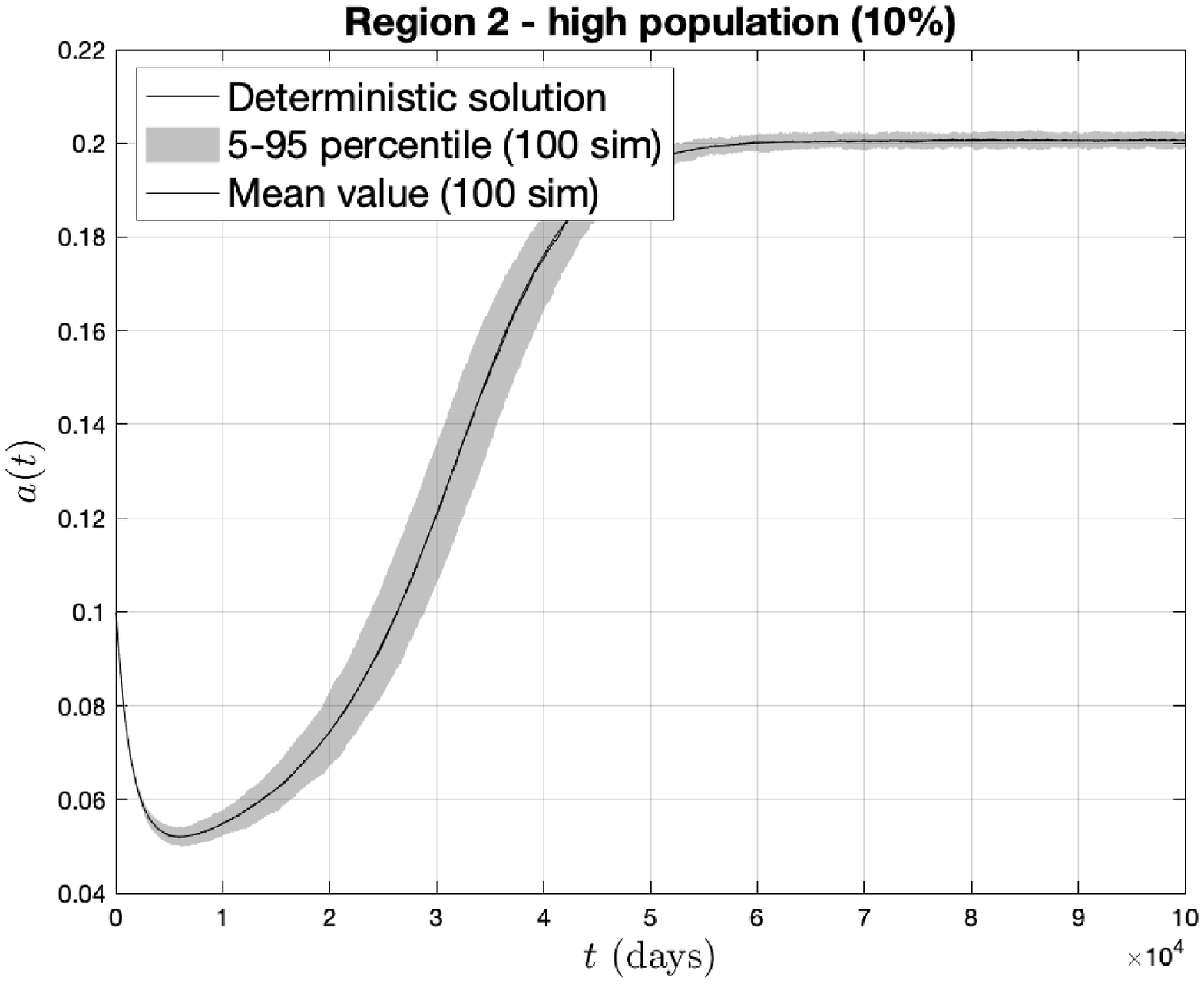}
\end{center}
\caption{Addicted population time series for Region 2 of the forward-backward bifurcation, with parameters $\mu=0.00015$, $\beta=0.009$, $\gamma=0.0027$, $\kappa=0.3111$, $\nu=0.8$, and $\phi=0.0044$, $R_0=0.9824$, $R_\phi=1.5439$. We present the mean $I(t)/N(t)$ for the stochastic model (100 simulations) for (a) low, (b) medium, and (c) high populations, with $I(0)/N(0)=0.01$ (top) and $I(0)/N(0)=0.10$ (bottom). Gray shaded region correspond to the 5th and 95th percentiles. Blue lines correspond to the deterministic solution.\label{stochasticModel2}}
\end{figure}

Figure~\ref{stochasticModel3} (region 3) shows that the addicted population will establish itself at either a large or a small size, once again depending on the initial addicted population size. Figures~\ref{stochasticModel2} and \ref{stochasticModel3} illustrate sensitivity to initial conditions for the deterministic model; whether the addicted community is established, and how large the community is, is dependent on how pervasive the initial population of addicted individuals is. For medium and large populations, the stochastic models preserves the qualitative behaviour of the deterministic curves. Nevertheless, in the low population cases, there are cases where a addicted-free state is reached.

\begin{figure}[htb!]
\begin{center}
\includegraphics[width=.3\linewidth]{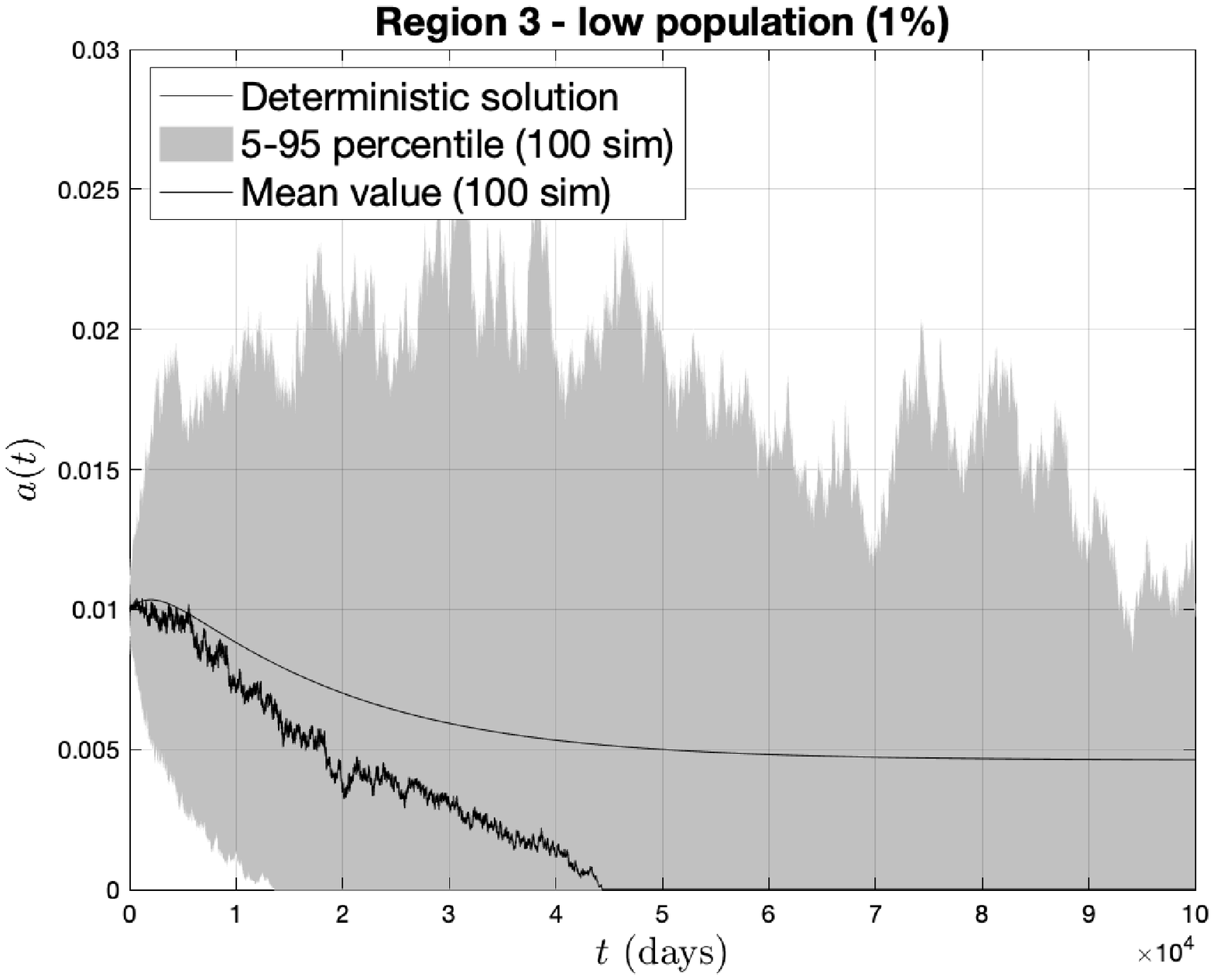}
\includegraphics[width=.3\linewidth]{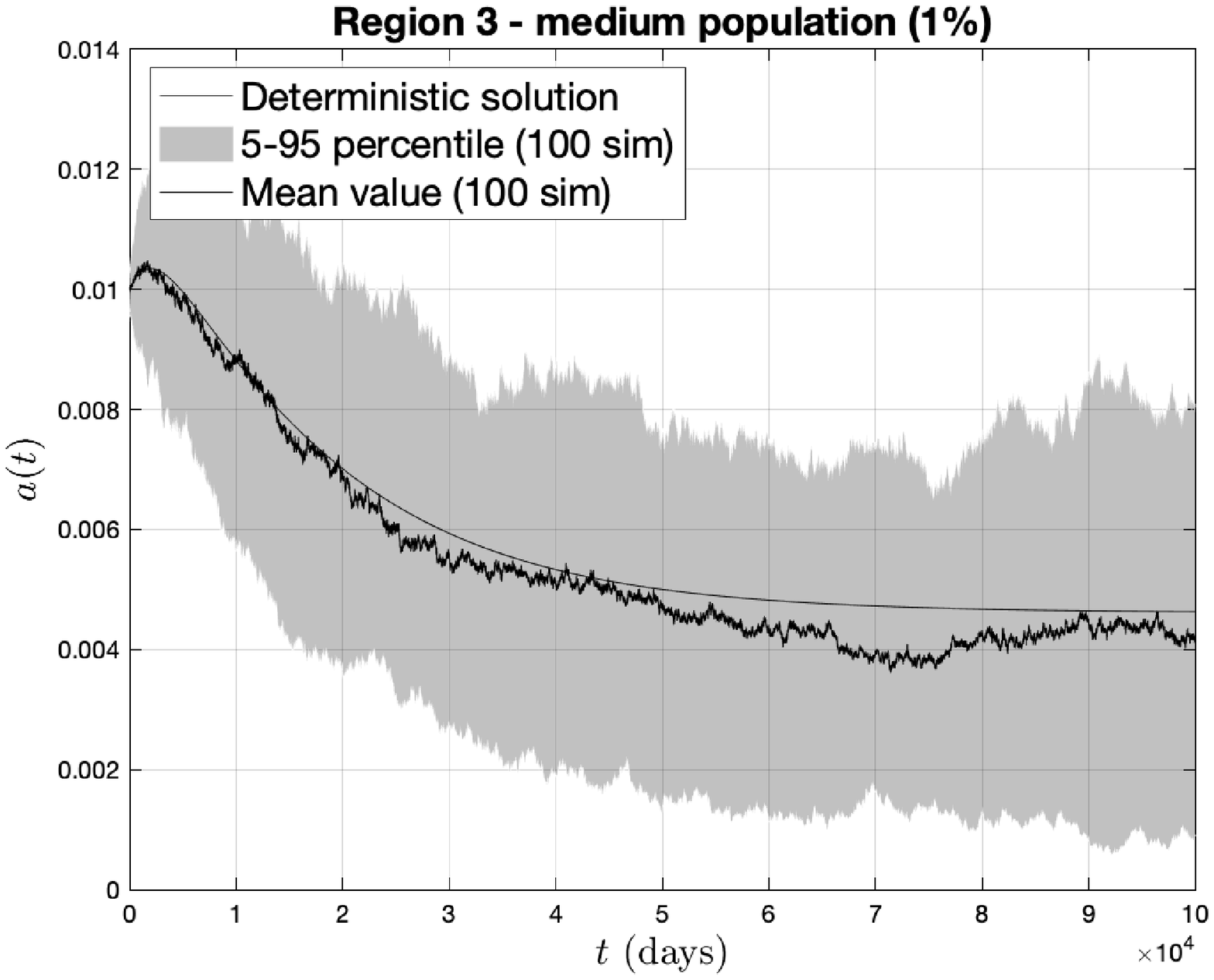}
\includegraphics[width=.3\linewidth]{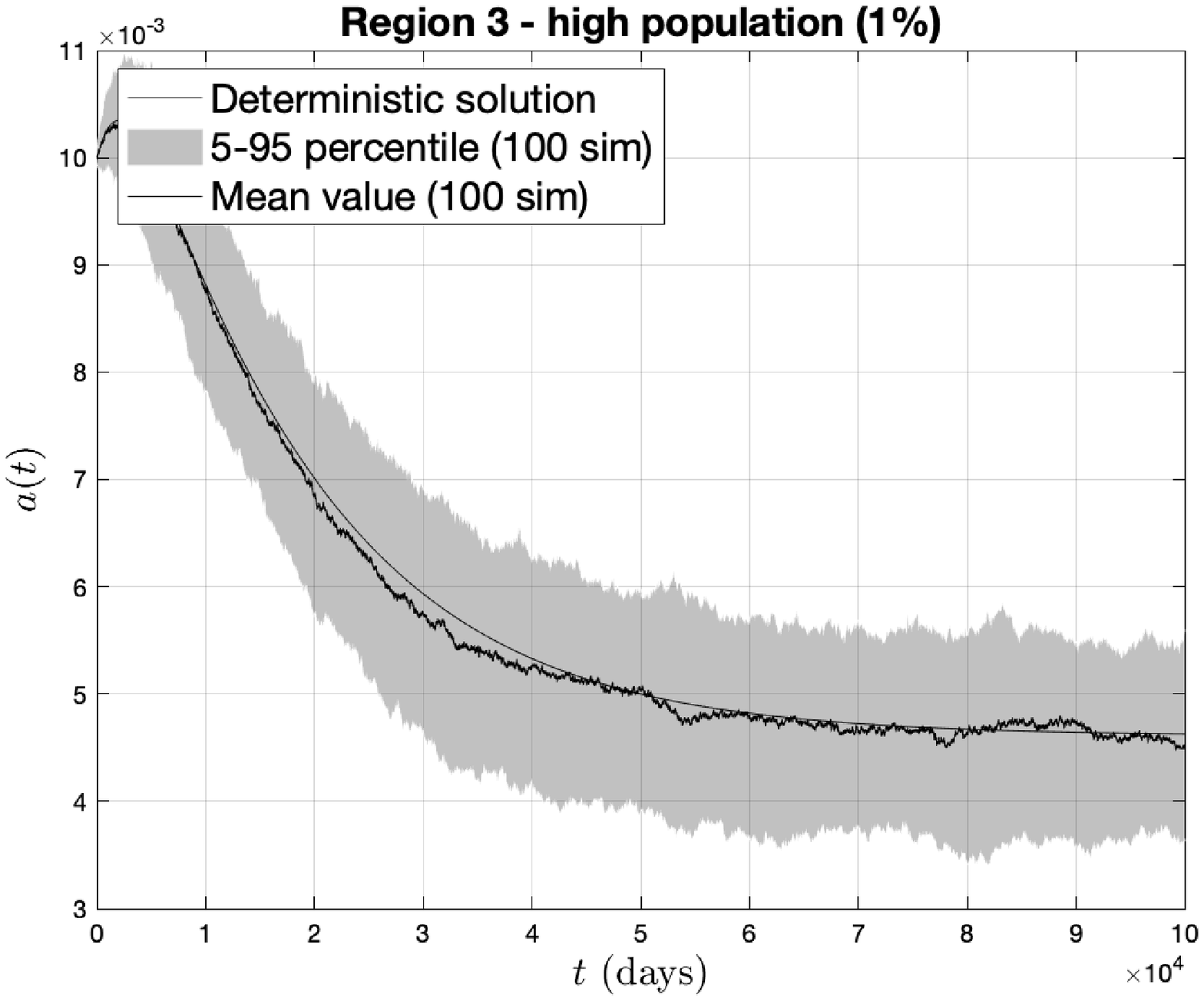}
\includegraphics[width=.3\linewidth]{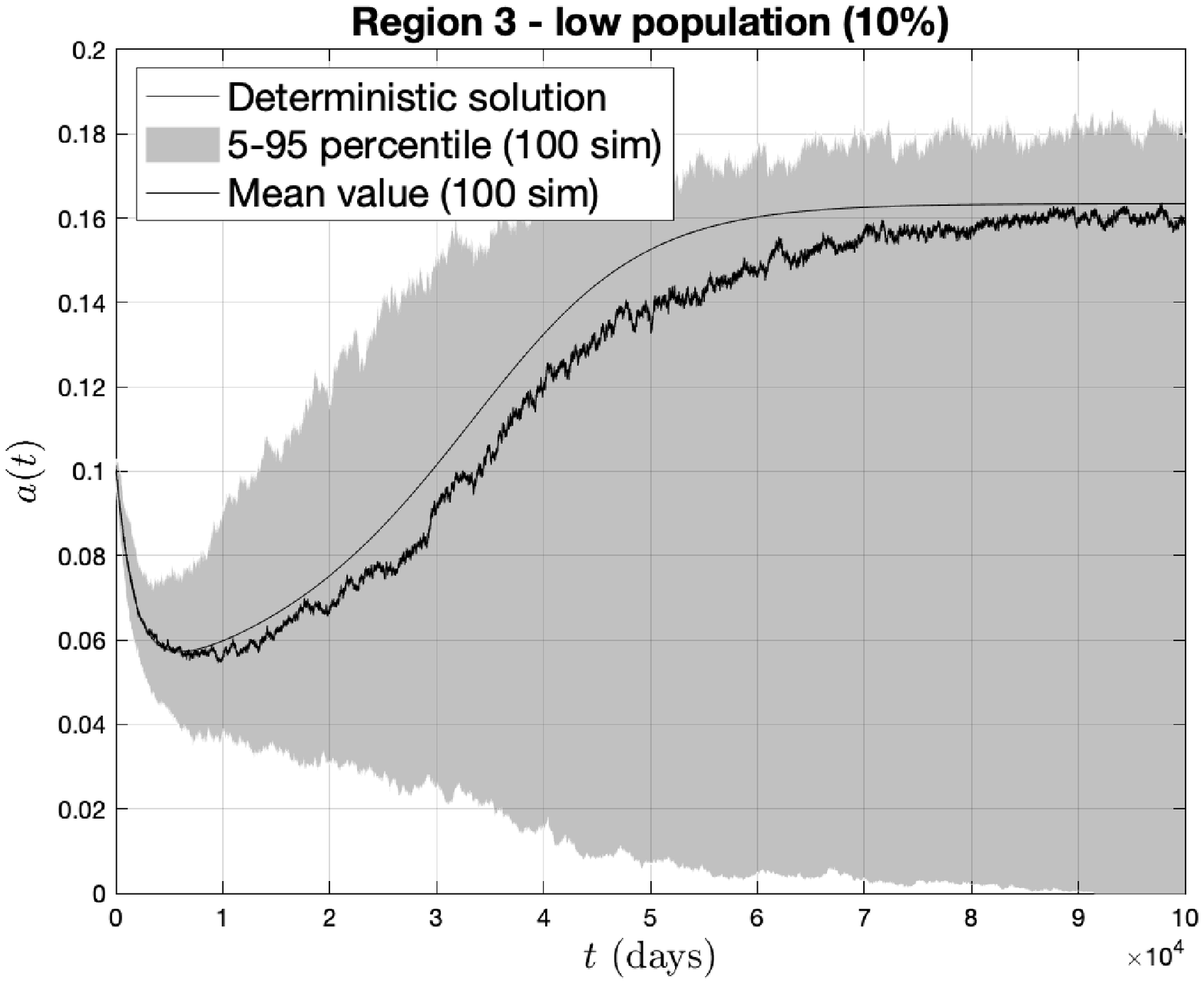}
\includegraphics[width=.3\linewidth]{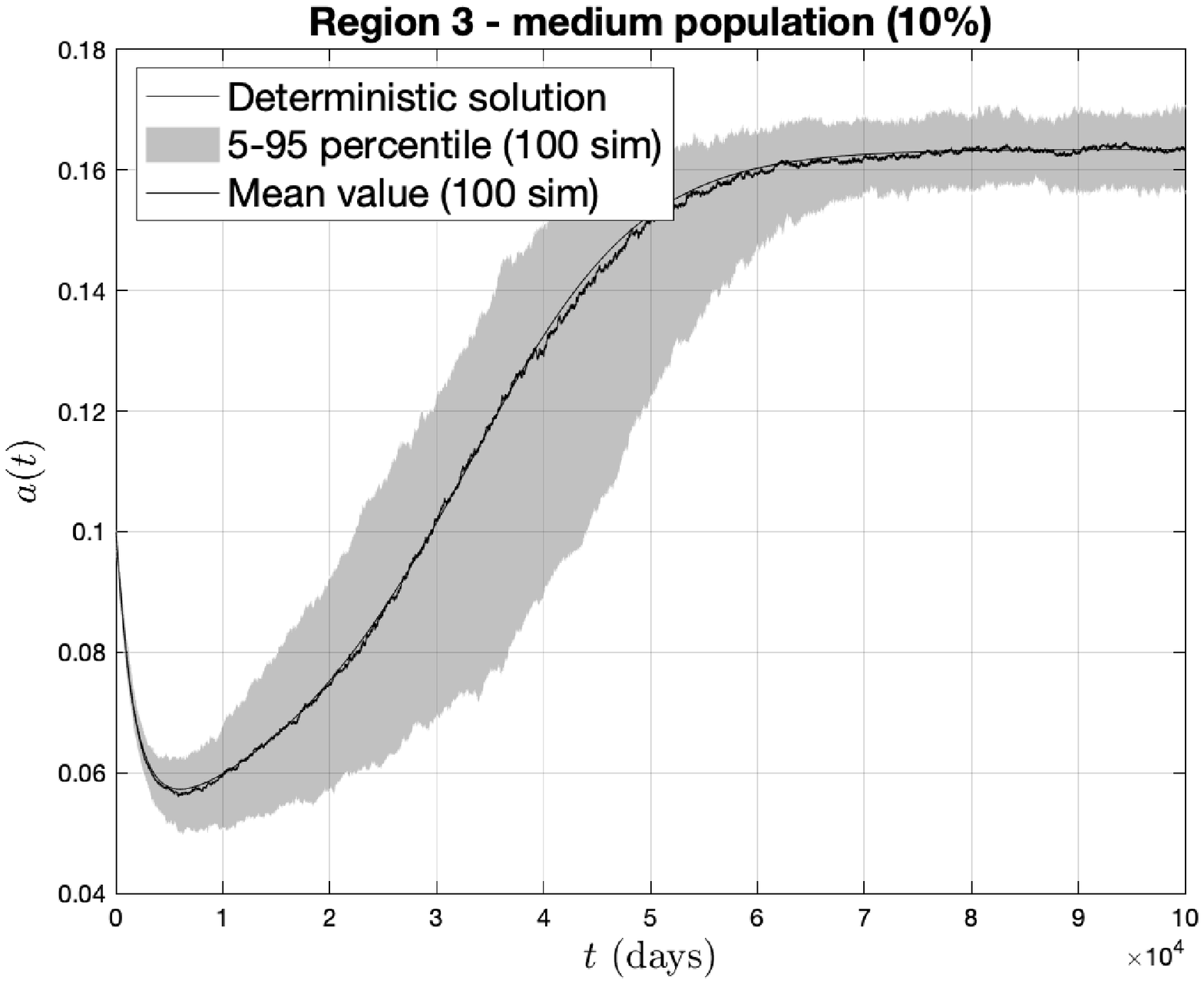}
\includegraphics[width=.3\linewidth]{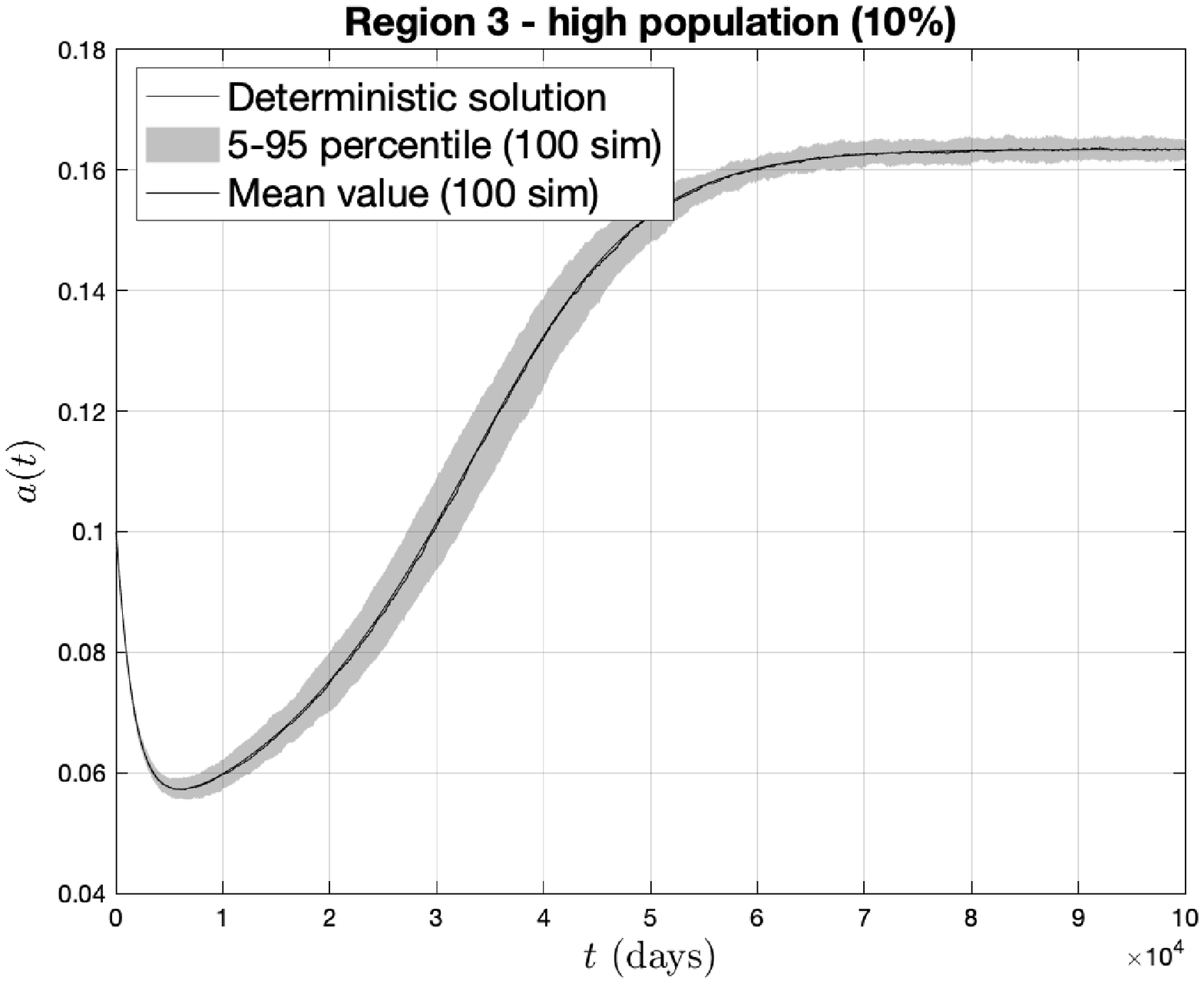}
\end{center}
\caption{Addicted population time series for Region 3 of the forward-backward bifurcation, with parameters $\mu=0.00015$, $\beta=0.009$, $\gamma=0.0027$, $\kappa=0.3243$, $\nu=0.8$, $\phi=0.0042$, $R_0=1.0241$, $R_\phi=1.4736$. We present the mean $I(t)/N(t)$ for the stochastic model (100 simulations) for (a) low, (b) medium, and (c) high populations, with $I(0)/N(0)=0.01$ (top) and $I(0)/N(0)=0.10$ (bottom). Gray shaded region correspond to the 5th and 95th percentiles. Blue lines correspond to the deterministic solution. \label{stochasticModel3}}
\end{figure}

Finally, Figure~\ref{stochasticModel4} (region 4) shows that the addicted population will establish itself at a (relatively) large population size despite a very small initial addicted population for the deterministic model. This occurs largely because the effects of relapse outweigh the attempts to discourage addicted involvement. Again, if the total population is small enough, an addicted-free state is possible, even though the mean of 100 simulations is close to the deterministic curve.

\begin{figure}[htb!]
\begin{center}
\includegraphics[width=.3\linewidth]{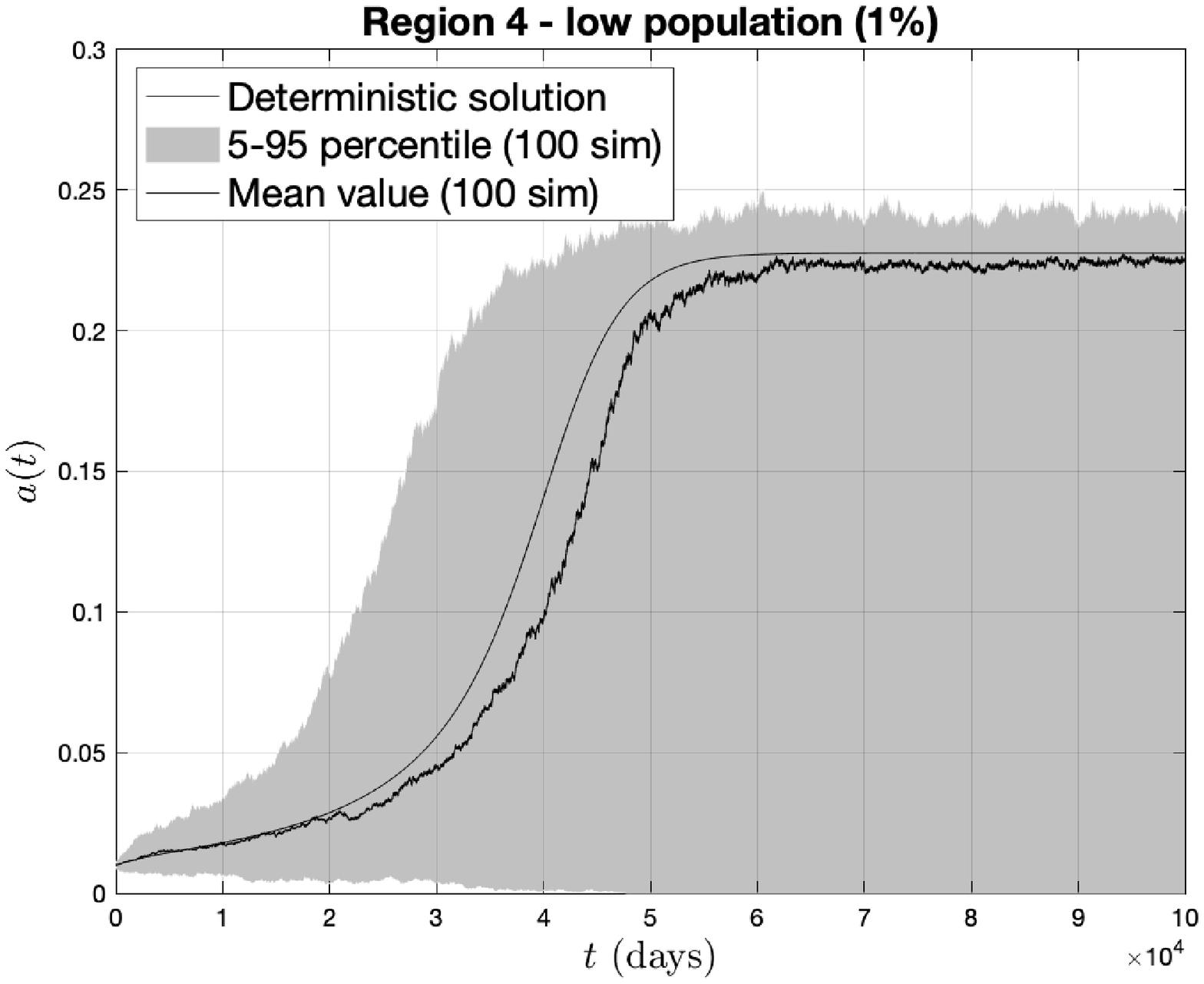}
\includegraphics[width=.3\linewidth]{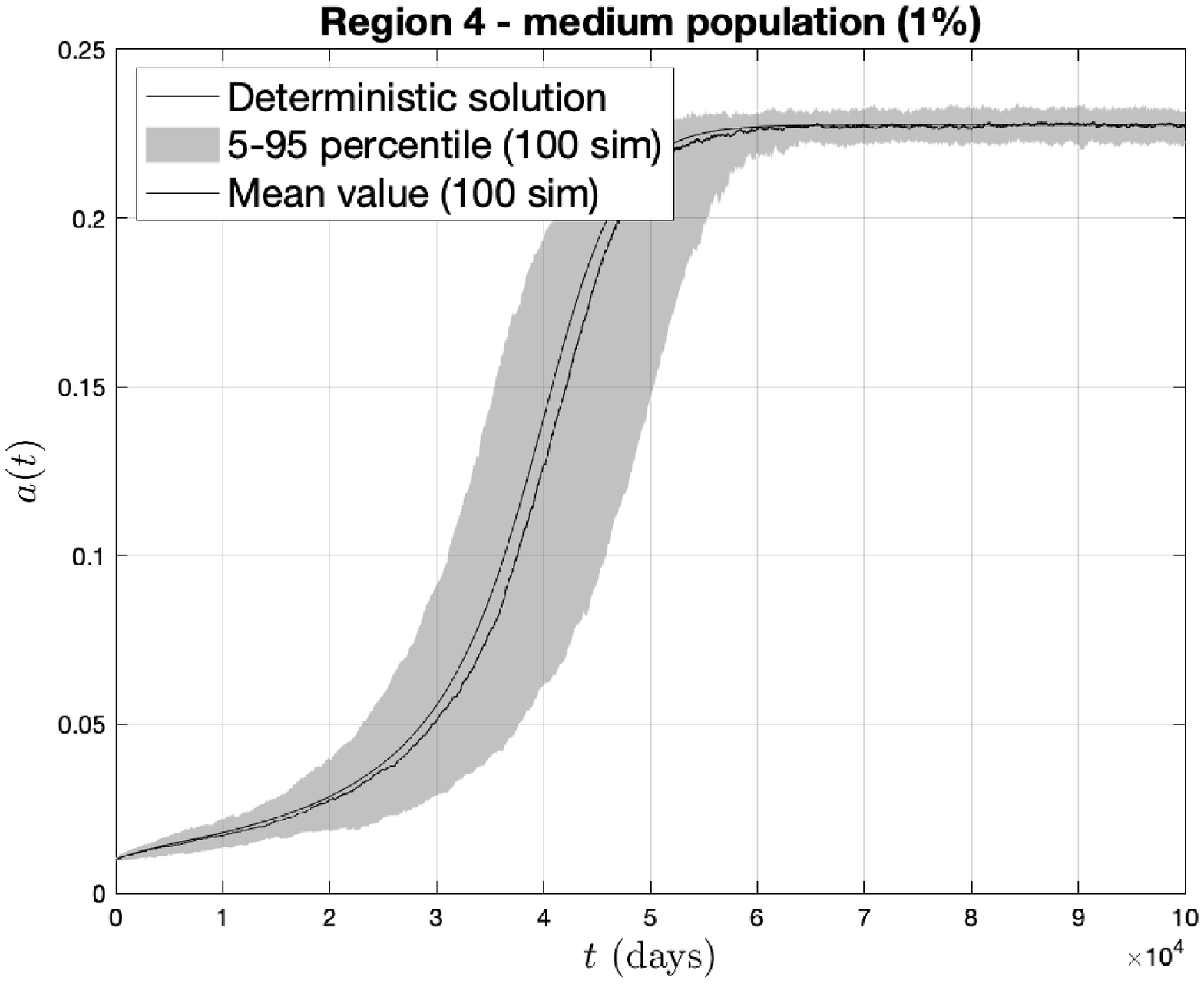}
\includegraphics[width=.3\linewidth]{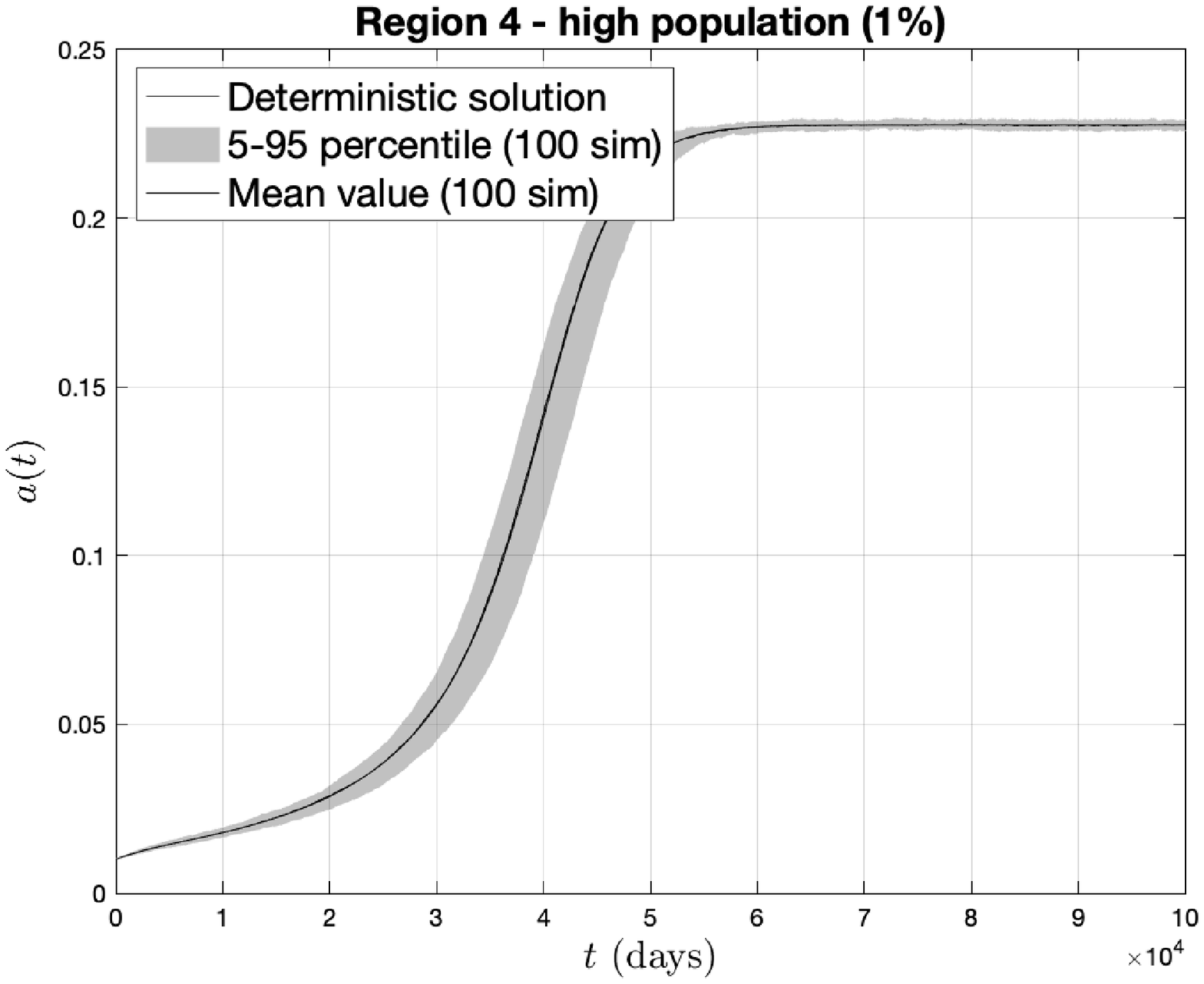}
\end{center}
\caption{Addicted population time series for Region 4 of the forward-backward bifurcation, with parameters $\mu=0.00015$, $\beta=0.009$, $\gamma=0.0027$, $\kappa=0.3333$, $\nu=0.8$, $\phi=0.0044$, $R_0=1.0525$, $R_\phi=1.5439$, $I(0)/N(0)=0.01$. We present the mean $I(t)/N(t)$ for the stochastic model (100 simulations) for (a) low, (b) medium, and (c) high populations. Gray shaded region correspond to the 5th and 95th percentiles. Blue dots correspond to the deterministic solution.\label{stochasticModel4}}
\end{figure}

\section{Discussion} \label{sec:disc}
We explored an epidemic-type model that includes nonlinear relapse in the temporarily recovered population. Results showed high sensitivity to the initial addicted population. If the initial addicted population in an at-risk environment is large enough, in other words, already established, then the addicted population is more likely to establish itself at an endemic equilibrium. Hence, well-established addicted populations play prominent social, political, and economic roles in the community and are thus much more difficult to control. 
Our model indicates that reformed individuals play a crucial role in addicted population dynamics when relapse rates are low. A large proportion of reformed individuals serving as mentors has the potential to significantly reduce the addicted population, given that relapse rates remain under control. A lack of opportunities could lead reformed individuals to return to addiction. Rehabilitation programs that aim to reintegrate reformed individuals into productive individuals in society must offer consistent supervised rehabilitation.  

While focusing on reformed individuals is essential in reducing the addicted population, recruitment into the addicted population is a major player in reducing the addicted population. The cost of becoming addicted plays a role in the {\it basic reproductive number}, $\mathcal{R}_0$, which implies that changes in $\kappa$ have significant implications in the transmission dynamics of the system. Furthermore, informing susceptible individuals about other lifestyles and opportunities, such as education, can discourage individuals from getting involved in risky environments where they can ultimately be pulled into addiction.  

Moreover, if the social influence of the reformed individuals on the {\it at-risk} susceptible population is strong, the long-term addicted population becomes manageable. For a specific cost ($\kappa$), a small region exists where $1<\mathcal{R}_0<\mathcal{R}_0^*$ and $\mathcal{R}_\phi>1$ with multiple stable addicted populations that are highly dependent on the initial addicted population size. If the initial addicted population in an {\it at-risk} environment is large enough, in other words, already established, then the addicted population establishes itself at the higher endemic equilibrium. Our model shows the influence of established problem communities, and it highlights the importance of prevention programs and relapse rates. 

While reformed individuals may impact addicted population dynamics, other factors, such as cost and relapse rate, play a role in the effectiveness of reformed individuals in population control. When relapse rates are low, reformed individuals play a crucial role in addicted population dynamics. A high value of $\nu$ can shift the forward-backward bifurcation to the point where $\mathcal{R}_0<1$ produces an addicted-free equilibrium. This highlights the importance of keeping relapse rates under control and encouraging reformed individuals to become involved with addiction prevention programs. 

From our model, we also found that the cost of addiction significantly impacts the addicted population dynamics. Cost is a factor in the basic reproductive number $\mathcal{R}_0$, which means that changes in this value may have significant implications for the addicted population. If costs are low to get into addiction, there is little that reformed individuals can do to decrease the growth in the addicted population. An alternative to lowering the cost of addiction is to educate {\it at-risk} individuals about the costs of addiction. This, in turn, may encourage individuals to look at these costs as a deterrent, which can ultimately help decrease the addicted population. Finally, more considerable efforts are needed not only to encourage reformed individuals to mentor individuals in an {\it at-risk} environment but also to reduce the relapse rate and help to educate the {\it at-risk} population to help contain individuals from getting into addiction.  

\section*{Acknowledgement(s)}

The authors would like to thank support by the Research Center in Pure and Applied Mathematics and the Department of Mathematics at Universidad de Costa Rica.

\section*{Disclosure statement}

All authors declare no conflicts of interest in this paper.





\end{document}